\documentclass [11pt] {amsart}

\usepackage{amsmath,amsthm, amsfonts, amssymb}
\usepackage{setspace}
\usepackage{color}
\usepackage{enumerate}
      
      \newtheorem{remark}{Remark}
      \newtheorem{theorem}{Theorem}
      \newtheorem{corollary}{Corollary}
      \newtheorem{lemma}{Lemma}
      
      \newtheorem{proposition}{Proposition}
      
      \def\N{{\mathbb N}}
      \def\C{{\mathbb C}}

      \def\cFdd'{\mathcal F_{dd'}}

      \def\cA{\mathcal A}

      \def\cB{\mathcal B}

\begin{document}
\newcommand\bigzero{\makebox(0,0){\text{\huge0}}}
\allowdisplaybreaks
\title{On the Gr\"{u}ss inequality for unital $2$-positive linear maps.}
\author[Balasubramanian]{Sriram Balasubramanian}
\address{Department of Mathematics\\
  IIT Madras, Chennai - 600036, India.}
\email{bsriram@iitm.ac.in, bsriram80@yahoo.co.in}



\maketitle
%

\begin{abstract}
In a recent work, Moslehian and Raji\'{c} have shown that the Gr\"{u}ss inequality holds for unital $n$-positive
linear maps $\phi:\cA \rightarrow B(H)$, where $\cA$ is a unital $C$*-algebra and $H$ is a Hilbert space, if $n \ge 3$. They also demonstrate that the inequality fails to hold, in general, if $n = 1$ and question whether the inequality holds if $n=2$. In this article, we provide an
affirmative answer to this question.
%
\end{abstract}
\begin{section}{Introduction}

A classical theorem of Gr\"{u}ss (see \cite{G}) states that if $f$ and $g$ are bounded real valued integrable functions on $[a,b]$ and $m_1 \le f(x) \le M_1$ and $m_2 \le g(x) \le M_2$ for all $x \in [a,b]$, then
\begin{equation*}
\left \vert \frac{1}{b-a} \displaystyle \int_a^b f(x) g(x) dx - \frac{1}{(b-a)^2} \left( \displaystyle \int_a^b f(x) dx \right) \left(\displaystyle \int_a^b g(x) dx\right) \right \vert
 \le \frac 1 4 \alpha \beta,
\end{equation*}
where $\alpha = (M_1-m_1)$ and $\beta = (M_2 - m_2)$.\\

A generalized operator version of the Gr\"{u}ss inequality was given by Renaud in \cite{R}, where he proved the following result.

\begin{theorem}
 Let $A, B \in B(H)$ and suppose that their numerical ranges are contained in disks of radii $R$ and $S$  respectively. If $T \in B(H)$ is a positive operator with Tr(T) = $1$, where $Tr$ stands for the trace, then
\[
|Tr(T AB) - Tr(T A)\,Tr(T B)| 􏰀\le 4RS.
\]
If $A, B$ are normal, then the constant $4$ on the right hand side can be replaced by $1$.
\end{theorem}

Among other operator versions of the Gr\"{u}ss inequality, of particular interest to us are those of Peri\'{c} and Raji\'{c} (see \cite{PR}), where they prove the Gr\"{u}ss inequality for completely bounded maps, and Moslehian and Raji\'{c} (see \cite{MR}), where they prove the Gr\"{u}ss inequality for $n$-positive unital linear maps, for $n \ge 3$. In \cite{MR}, the authors show that the inequality fails to hold in general, if $n=1$ and question whether it holds for the case $n=2$. The main result of this article gives an affirmative answer to this question.

Before we state the main result, we shall introduce some notation and definitions. Throughout this article, $\cA$ will denote a unital $C$*-algebra, $M_n(\cA)$ the $C$*-algebra of $n \times n$ matrices over $\cA$, $H$ and $K$ complex Hilbert spaces and $B(H)$ the $C$*-algebra of bounded operators on $H$. The notations $e$, $1$ will denote the unit elements in $\cA$ and $B(H)$ respectively and $\phi:\cA \rightarrow B(H)$, a unital linear map, i.e. a linear map such that $\phi(e) = 1$. The map $\phi$ is said to be positive if $\phi(a)$ is positive in $B(H)$ for all positive $a \in \cA$.  For more details, see \cite{P}.  It is easy to see that the map $\phi_n: M_n(\cA) \rightarrow M_n(B(H))$ defined by $\phi_n((a_{ij})) = (\phi(a_{ij}))$ is unital and linear for each $n \in \N$. The map $\phi$ is said to be $n$-positive if $\phi_n$ is a positive map, completely positive if $\phi$ is $n$-positive for all $n \in \N$ and completely bounded if $\sup_{n \in \N} \|\phi_n\|  < \infty$. \\\\
The main result of this article is the following.
\begin{theorem}
\label{thm:main}
Let $\cA$ be a $C$*-algebra with unit $e$.
If $\phi: \cA \rightarrow B(H)$ is a unital $2$-positive linear map, then
\begin{equation}
\label{eq:main}
\|\phi(ab) - \phi(a) \phi(b) \| \le \left(\inf_{\lambda \in \C} \|a - \lambda e\| \right) \left(\inf_{\mu \in \C} \|b - \mu e\| \right).
\end{equation}
for all $a,b \in \cA$.
\end{theorem}

To prove Theorem \ref{thm:main}, we use the well-known theorems of Stinespring, Russo-Dye, Fuglede-Putnam, and the result due to Choi (see Lemma \ref{lem:choi}).

\end{section}

\section{Preliminaries}
In this section we include some lemmas which will be used in the sequel.
Observe that if $\cA$ and $\mathcal B$ are unital $C$*-algebras and $\gamma:\cA \rightarrow \mathcal B$ is a unital $n$-positive linear map, then it is $m$-positive for all $m = 1, 2, \dots, n$. In particular $\gamma$ is positive. It is well known that positive maps are *-preserving. i.e. $\gamma(a^*) = \gamma(a)^*$ for all $a \in \cA$. Moreover $\|\gamma\| = 1.$

\begin{lemma}
\label{lem:2by2pos}
If $P,Q,R \in B(H)$, then $A = \begin{pmatrix}  P & R \\ R^* & Q \end{pmatrix}  \succeq 0$ in $M_2(B(H))$ if and only if $P,Q \succeq 0$ and $| \langle Rx,y \rangle|^2 \le \langle Px,x \rangle \langle Qy,y \rangle$, for
all $x,y \in H$. Moreover, if $A \succeq 0$, then $\|R\|^2 \le \|P\| \|Q\|$.
\end{lemma}

\begin{lemma}
\label{lem:blockpos}
Let $A = \begin{pmatrix}  T & S \\ S^* & R \end{pmatrix} \in B(H \oplus K)$. If $R \in B(K)$ be invertible, then the following statements are equivalent.
\begin{itemize}
\item [(i)] $A \succeq 0$
\item [(ii)] $T, R \succeq 0$ and $T \succeq SR^{-1}S^*$.
\end{itemize}
\end{lemma}
The above two lemmas are well known. Their proofs can be found in \cite{A}.

\begin{lemma}[Choi]
\label{lem:choi}
Let $\mathcal U$ and $\mathcal V$ be $C$*-algebras and $\phi: \mathcal U \rightarrow \mathcal V$ be a positive linear map. If $x,y \in \mathcal U$ and
$
\begin{pmatrix}
x & y \\  y^* & x
\end{pmatrix} \succeq 0,
$
then
$
\begin{pmatrix}
\phi(x) & \phi(y) \\  \phi(y^*) & \phi(x)
\end{pmatrix} \succeq 0.
$
\end{lemma}

For a proof of Lemma \ref{lem:choi}, please see Corollary 4.4 of \cite{C}.

\begin{proposition}
\label{prop:2and3pos}
If $\mathcal B$ is a unital $C$*-algebra and $\phi:\mathcal B \rightarrow B(H)$ is a unital $2$-positive linear map, then
\begin{equation}
\label{eq:impineq}
\| \phi(ab) - \phi(a) \phi(b) \|^2 \le \| \phi(aa^*) - \phi(a) \phi(a)^* \| \| \phi(b^*b) - \phi(b)^*\phi(b) \|,
\end{equation}
for all unitaries $a,b \in \mathcal B$.
\end{proposition}
\begin{proof}
Since $\phi$ is positive, recall that $\phi(x^*) = \phi(x)^*$ for all $x \in \cB$. Let $a,b \in \mathcal B$ be unitary. Consider the matrix
\[
A =
 \left( \begin{array}{ccc | c}
a^*a & a^*b & a^* & a^*(a^*b) \\ b^*a & b^*b & b^* & b^*(a^*b)\\ a & b & a^*a & a^*b \\ \hline
(b^*a)a & (b^*a)b & b^*a & b^*b \end{array} \right).
\]
Since $a, b$ are unitaries, it follows that $R= b^*b = e$ and
\[
T =
\begin{pmatrix}  a^*a & a^*b & a^* \\ b^*a & b^*b & b^*\\ a & b & a^*a \end{pmatrix} = \begin{pmatrix}
a^*(a^*b) \\   b^*(a^*b) \\  a^*b \end{pmatrix} \begin{pmatrix}
(b^*a)a & (b^*a)b & b^*a \end{pmatrix} = SS^* = SR^{-1}S^*.
\]
Thus Lemma \ref{lem:blockpos} implies that $A \succeq 0$. This is equivalent to
\begin{equation}
\label{eq:block}
\left( \begin{array}{cc | cc}
a^*a & a^*b & a^* & a^*(a^*b) \\ b^*a & b^*b & b^* & b^*(a^*b)\\ \hline a & b & a^*a & a^*b \\
(b^*a)a & (b^*a)b & b^*a & b^*b \end{array} \right) \succeq 0.
\end{equation}
By Lemma \ref{lem:choi} applied to the unital positive map $\phi_2$ and the $2 \times 2$ block matrix in equation \eqref{eq:block}, it follows that
\begin{align}
&\left( \begin{array}{cccc}
\phi(a^*a) & \phi(a^*b) & \phi(a)^* & \phi(a^*(a^*b)) \\ \phi(b^*a) & \phi(b^*b) & \phi(b)^* & \phi(b^*(a^*b))\\ \phi(a) & \phi(b) & \phi(a^*a) & \phi(a^*b) \\
\phi((b^*a)a) & \phi((b^*a)b) & \phi(b^*a) & \phi(b^*b) \end{array} \right)
\succeq 0.
\end{align}
This in turn implies that
\begin{equation}
\label{eq:imp}
\begin{pmatrix}
\phi(a^*a) & \phi(a^*b) & \phi(a)^* \\
\phi(b^*a) & \phi(b^*b) & \phi(b)^* \\
\phi(a) & \phi(b) & \phi(a^*a)
\end{pmatrix}
\succeq 0.
\end{equation}
By Lemma \ref{lem:blockpos} and the fact that $\phi(a^*a) = \phi(e) = 1$, equation \eqref{eq:imp} is equivalent to
\begin{equation}
\begin{pmatrix}
\phi(a^*a) & \phi(a^*b) \\
\phi(b^*a) & \phi(b^*b)
\end{pmatrix}
-
\begin{pmatrix}
\phi(a)^* \\
\phi(b)^*
\end{pmatrix}
\begin{pmatrix}
\phi(a) & \phi(b)
\end{pmatrix}
\succeq 0,
\end{equation}
i.e.
\begin{equation}
\label{eq:vimp}
\begin{pmatrix}
\phi(a^*a)-\phi(a)^* \phi(a) & \phi(a^*b)-\phi(a)^*\phi(b) \\
 \phi(b^*a)-\phi(b)^*\phi(a) & \phi(b^*b)-\phi(b)^*\phi(b)
\end{pmatrix}
\succeq 0.
\end{equation}
An application of Lemma \ref{lem:2by2pos} to the operator matrix in equation \eqref{eq:vimp} yields
\begin{equation}
\label{eq:normineq}
\|\phi(a^*b) - \phi(a)^* \phi(b)\|^2 \le \|\phi(a^*a) - \phi(a)^* \phi(a) \| \| \phi(b^*b) - \phi(b)^* \phi(b) \|
\end{equation}
for all unitaries $a, b \in \cB$. Replacing $a$ by $a^*$ in \eqref{eq:normineq} completes the proof.
\end{proof}

The following three theorems are well known.


\begin{theorem}[Fuglede-Putnam]
\label{thm:fp}
Let $\cA$ be a $C$*-algebra. If $x,y \in \cA$ are such that $x$ is normal and $xy=yx$, then $x^*y = yx^*$.
\end{theorem}

For more on the Fuglede-Putnam theorem, please see \cite{B}.

\begin{theorem}[Stinespring's Dilation Theorem]
\label{thm:stinedil}
If $\mathcal B$ is a unital $C$*-algebra and $\phi:\mathcal B \rightarrow B(H)$ is a unital completely positive map, then there exist a Hilbert space $K$, an isometry $v: H \rightarrow K$ and a unital *-homomorphism $\pi: \cB \rightarrow B(K)$ such that $\phi(x) = v^* \pi(x) v$ for all $x \in \cB$.
\end{theorem}

For a proof of Stinespring's dilation theorem, please see \cite{P}.

\begin{theorem}[Russo-Dye]
\label{thm:rd}
Let $\cA$ be a unital $C$*-algebra. If $a \in \cA$ is such that $\|a\| < 1$, then $a$ is a convex combination of unitary elements in $\cA$.
\end{theorem}

For a proof and more on the Russo-Dye theorem, please see \cite{B}.

\section{The Proof}
\label{sec:proof}

This section contains the proof of our main result, i.e. Theorem \ref{thm:main}. The following theorem and corollary lead us to it.

\begin{theorem}
\label{thm:stinespring}
If $a,b$ are commuting normal elements in the unital $C$*-algebra $\cA$ and $\phi: \cA \rightarrow B(H)$ is a unital positive linear map, then
\begin{equation}
\label{eq:commnormal}
\| \phi(ab) - \phi(a) \phi(b) \| \le \left(\inf_{\lambda \in \C} \|a - \lambda e\| \right) \left(\inf_{\mu \in \C} \|b - \mu e\| \right),
\end{equation}
i.e. the Gr\"{u}ss inequality holds for such $a,b \in \cA$.
\end{theorem}
\begin{proof}
The proof is adapted from \cite{PR}. Let $\lambda, \mu \in \C$. Since $a,b$ are commuting normal elements in the $C$*-algebra $\cA$, it follows from Theorem \ref{thm:fp} that the $C$*-subalgebra of $\cA$, say $\cB$, generated by $a, b$ and $e$ is commutative. Since the restricted map $\phi: \cB \rightarrow B(H)$ is positive and $\cB$ is commutative, it follows that $\phi: \cB \rightarrow B(H)$ is in fact completely positive (see e.g. \cite{P}). By Theorem \ref{thm:stinedil}, it follows that there exist a Hilbert space $K$, an isometry $v: H \rightarrow K$ and a unital *-homomorphism $\pi: \cB \rightarrow B(K)$ such that $\phi(x) = v^* \pi(x) v$ for all $x \in \cB$. Since $\pi$ is a unital *-homomorphism, it is completely positive and hence is a complete contraction. In particular $\| \pi \| \le 1$. It follows that
\begin{align*}
\| \phi(ab) - \phi(a) \phi(b) \|  & =  \|\phi((a - \lambda e)(b - \mu e)) - \phi(a - \lambda e) \phi(b - \mu e)\| \\
                & = \| v^* \pi((a - \lambda e)(b - \mu e)) v - v^* \pi(a - \lambda e) vv^* \pi(b - \mu e) v \|\\
                & = \| v^* \pi(a - \lambda e) \pi(b - \mu e) v - v^* \pi(a - \lambda e) vv^* \pi(b - \mu e) v \|\\
                & = \|v^* \pi(a - \lambda e) (1 - vv^*) \pi(b - \mu e) v\| \\
                & \le \|a - \lambda e\| \|1 - vv^*\| \|b - \mu e\|\\
                & \le \|a - \lambda e\| \|b - \mu e\|.
\end{align*}
The proof is complete by taking infimums on the above inequality first with respect to $\lambda$ and then with respect to $\mu$.
\end{proof}

\begin{remark}
\label{rem:cp}
It is easy to see that if $\cA$ is commutative or $\phi$ is completely positive, in the statement of Theorem \ref{thm:stinespring}, then the entire proof of Theorem \ref{thm:stinespring} goes through with $\cB$ replaced by $\cA$, for arbitrary $a$ and $b$, i.e. the Gr\"{u}ss inequality \eqref{eq:commnormal} holds if $\cA$ is commutative or $\phi$ is completely positive.
\end{remark}

\begin{corollary}
\label{cor:normal}
If $\phi$ and $a$ are as in Theorem \ref{thm:stinespring}, then
\[
\|\phi(aa^*) - \phi(a) \phi(a)^*\| \le \left(\inf_{\lambda \in \C} \|a - \lambda e\| \right)^2.
\]
\end{corollary}
\begin{proof}
The proof follows by taking $b = a^*$ in Theorem \ref{thm:stinespring}.
\end{proof}

\begin{proof}[ \textbf{Proof of Theorem \ref{thm:main}}]
Recall $a$, $b$, $\cA$, $H$ and $\phi$ from the statement of Theorem \ref{thm:main}. Let $\epsilon > 0$. By Theorem \ref{thm:rd}, there exist unitary elements $u_1, \dots, u_k$ and $v_1, \dots, v_{\ell}$ in $\cA$ such that $\frac{a}{(\|a\| + \epsilon)} =  \sum_{i = 1}^k \alpha_i u_i$ and $\frac{b}{(\|b\| + \epsilon)} =  \sum_{j = 1}^{\ell} \beta_j v_j$, where $\alpha_i, \beta_j \ge 0$ and $ \sum_{i = 1}^k \alpha_i = \sum_{j=1}^{\ell} \beta_j = 1$. It follows from
Proposition \ref{prop:2and3pos} and Corollary \ref{cor:normal} that
\begin{align}
\frac{1}{(\|a\| + \epsilon)} & \frac{1}{(\|b\| + \epsilon)} \| \phi(ab) - \phi(a) \phi(b)\| \nonumber \\
&= \left\Vert \phi \left(\left(  \sum_{i = 1}^k \alpha_i u_i\right) \left( \sum_{j = 1}^{\ell} \beta_j v_j \right) \right) - \phi \left( \displaystyle \sum_{i = 1}^k \alpha_i u_i \right) \phi \left(  \sum_{j = 1}^{\ell} \beta_j v_j\right) \right\Vert \nonumber\\
& \le \sum_{i=1}^k \sum_{j=1}^{\ell} \alpha_i \beta_j  \| \phi(u_iv_j) - \phi(u_i) \phi(v_j) \|  \\
&\le  \sum_{i=1}^k \sum_{j=1}^{\ell} \alpha_i \beta_j \|\phi(u_iu_i^*) - \phi(u_i) \phi(u_i)^*\|^{\frac 1 2}
\|\phi(v_j^*v_j) - \phi(v_j)^* \phi(v_j)\|^{\frac 1 2} \\
& \le  \sum_{i=1}^k \sum_{j=1}^{\ell} \alpha_i \beta_j \left(\inf_{\lambda \in \C} \|u_i - \lambda e\| \right) \left(\inf_{\mu \in \C} \|v_j - \mu e\| \right) \nonumber\\
& \le  \sum_{i=1}^k \sum_{j=1}^{\ell} \alpha_i \beta_j \|u_i\|\|v_j\| \nonumber\\
& = \sum_{i=1}^k \sum_{j=1}^{\ell} \alpha_i \beta_j \nonumber\\
& =  \, \left(\sum_{i=1}^k \alpha_i \right)  \left(\sum_{j=1}^{\ell} \beta_j \right) \nonumber \\
& = \, 1. \nonumber
\end{align}
Letting $\epsilon \rightarrow 0$ above yields,
\begin{equation}
\label{eq:norm}
\| \phi(ab) - \phi(a) \phi(b) \| \le \|a\| \|b\|.
\end{equation}
Let $\lambda, \mu \in \C$ be arbitrary. It follows from equation \eqref{eq:norm} that
\begin{align*}
\|\phi(ab) - \phi(a) \phi(b) \| &= \|\phi((a - \lambda e)(b - \mu e)) - \phi(a - \lambda e) \phi(b - \mu e)\| \\
                                           & \le \|(a - \lambda e)\| \|(b - \mu e)\|.
\end{align*}
Taking infimums in the above inequality, first with respect to $\lambda$ and then with respect to $\mu$ completes the proof.
%
\end{proof}


The Gr\"{u}ss inequality fails, in general, when $\phi$ in Theorem \ref{thm:main} is assumed only to be positive, i.e. when $n = 1$, as the following example shows. We point out that \cite{MR} contains an example of such a map $\phi: M_2(\C) \rightarrow M_2(\C)$.\\\\
\textbf{Example:}
Let $k \ge 2$, $\beta = \{e_1,e_2,\dots,e_k\}$ be an orthonormal set in $H$, $E$ = span$(\beta)$, and $\theta:M_k(\C) \rightarrow M_k(\C)$ denote the transpose map. It is well known that $\theta$ is a unital positive linear map, which is not $2$-positive (see \cite{TT}). Define $\phi:M_k(\C) \rightarrow B(H)$ by $\phi(a) = \begin{pmatrix}  \theta(a) & \textbf{0} \\ \textbf{0} & \textbf{1} \end{pmatrix}$. The block structure is with respect to the orthogonal decomposition $E \oplus E^{\perp}$ of $H$. Here $\textbf{1}$ denotes the identity operator and $\textbf{0}$ denotes the zero operator. It is easy to see that $\phi$ is a unital positive linear map which is not $2$-positive. Let $a= \begin{pmatrix} 1 & 3 \\ 3 & 3 \end{pmatrix}  \oplus \textbf{0}_{k-2} \in M_k(\C)$ and $b = \begin{pmatrix} 1 & 0 \\ 0 & 3 \end{pmatrix}  \oplus \textbf{0}_{k-2} \in M_k(\C)$. A simple computation shows that the eigenvalues of $a$ belong to $\{0, 2 \pm \sqrt{10}\}$ and those of $b$ belong to $\{0, 1, 3\}$. Since $a$ and $b$ are normal, it follows from \cite{S} that,
\begin{equation}
\label{eq:infs}
\inf_{\lambda \in \C} \|a - \lambda e\|  = \sqrt{10} \text{   \, \, \,  and  \, \, \,   }
 \inf_{\mu \in \C} \|b - \mu e\|  = \frac{3}{2}.
\end{equation}
Moreover
\begin{align*}
\phi(ab) - \phi(a) \phi(b) &= \left(\left( \begin{pmatrix}  1 & 3 \\ 9 & 9 \end{pmatrix} \oplus \textbf{0}_{k-2} \right) \oplus \textbf{1} \right)- \left(\left(\begin{pmatrix} 1 & 9 \\ 3 & 9 \end{pmatrix} \oplus \textbf{0}_{k-2} \right) \oplus \textbf{1} \right)
\\
&= \left(\left( \begin{pmatrix}  0 & -6 \\ 6 & 0 \end{pmatrix} \oplus \textbf{0}_{k-2} \right) \oplus \textbf{0} \right).
\end{align*}
Thus,
\[
\|\phi(ab) - \phi(a) \phi(b)\| = 6 >  \sqrt{10} \cdot \frac{3}{2}  = \left(  \inf_{\lambda \in \C} \|a - \lambda e\| \right) \left(   \inf_{\mu \in \C} \|b - \mu e\|  \right).
\]
\textbf{Acknowledgements:} I would like to thank Prof. Scott McCullough for some useful discussions and Krishanu Deyasi for his help with some MATLAB computations.


\begin{thebibliography}{99}
\bibitem[A]{A} T. Ando, Topics on Operator Inequalities, Division of Appl. Math., Research Institute
of Applied Electricity, Hokkaido Univ., Sapporo, 1978.
\bibitem[B]{B} B. Blackadar, Operator Algebras: Theory of $ C$*-Algebras and von Neumann Algebras, Encyclopaedia of Mathematical Sciences, Vol 122, 2006, ISBN 978-3-540-28517-5.
\bibitem[C]{C} M.D. Choi, Some Assorted Inequalities for Positive linear maps on $C$*-algebras, J. Oper. Th., 4 (1980) 271-285.
\bibitem[G]{G} G. Gr\"{u}ss, \"{U}ber das Maximum des absoluten Betrages von $\frac{1}{b-a} \displaystyle \int_a^b f(x) g(x) dx - \frac{1}{(b-a)^2} \displaystyle \int_a^b f(x) dx \displaystyle \int_a^b g(x) dx$, Math. Z. 39 (1934), 215-226.
\bibitem[MR]{MR} M.S. Moslehian, R. Raji\'{c}, A Gr\"{u}ss inequality for $n$-positive linear maps, Linear Algebra  Appl. 433 (2010), 1555-1560.
\bibitem[P]{P} V. Paulsen, Completely Bounded Maps and Operator Algebras,
 Cambridge University Press, 1st edition, 2003.
\bibitem[PR]{PR} I.Peri\'{c}, R.Raji\'{c}, Gr\"{u}ss inequality for completely bounded maps, Linear Algebra Appl. 390 (2004), 287-292.
\bibitem[R]{R} P.F. Renaud, A matrix formulation of Gr\"{u}ss inequality, Linear Algebra Appl. 335 (2001) 95-100.
\bibitem[S]{S} J.G. Stampfli, The norm of a derivation, Pacific J. Math. 33 (1970) 737-747.
\bibitem[TT]{TT} T. Takasaki, J. Tomiyama, On the geometry of positive maps in matrix algebras, Math. Z. 184 (1983), 101-108.
\end{thebibliography}
\end{document}